\documentclass[12pt]{article}
\usepackage{authblk}
\usepackage{hyperref}
\usepackage{amssymb}
\usepackage{amsthm}
\usepackage{amsmath}
\usepackage{multirow}
\usepackage{float}
\usepackage{graphicx}
\usepackage{listings}
\usepackage{MnSymbol}
\usepackage{nicematrix}
\usepackage{titlesec}
\usepackage{xcolor}
\newtheorem{theorem}{Theorem}[section]

\newtheorem{proposition}[theorem]{Proposition}

\newtheorem{example}[theorem]{Example}

\newtheorem{remark}[theorem]{Remark}

\begin{document}
\title{\Large\bfseries Linearization Problem for Third-Order ODEs with Four- and Five-Dimensional Lie Symmetry Algebras under Contact Transformations}

\author[1]{Omar A. Abuloha\thanks{1227001@student.birzeit.edu}}
\author[1]{Marwan Aloqeili\thanks{maloqeili@birzeit.edu}}
\author[1]{Ahmad Y. Al-Dweik\thanks{aaldweik@birzeit.edu}}
\author[2]{F. M. Mahomed\thanks{Corresponding Author: Fazal.Mahomed@wits.ac.za}}
\affil[1]{Department of Mathematics, Birzeit University, Ramallah, Palestine}
\affil[2]{School of Computer Science and Applied Mathematics, 
University of the Witwatersrand, Johannesburg, Wits 2050,  South Africa}
\maketitle
\begin{abstract}
Using Cartan’s equivalence method, invariant coframes are constructed for two branches of rank one and zero, which characterize linearizable third-order ODEs under contact transformations with four- and five-dimensional Lie symmetry algebras, respectively. A procedure for deriving the corresponding contact transformations is also presented, along with illustrative examples.
\end{abstract}
\section{Introduction}
 Lie established practical and algebraic linearization criteria for scalar second-order ordinary differential equations (ODEs) via invertible point transformations \cite{Lie, Lie*}. He proved that such linearizable equations possess a degree no higher than three with respect to the first derivative, while providing coefficient-based criteria and algebraic conditions for the reduction to linear form \cite{Lie}. A scalar second-order ODE admitting the maximal eight point symmetries is linearizable. Moreover, any equation admitting a rank-one, two-dimensional Lie algebra of point symmetries is linearizable by a point transformation (see, e.g., Mahomed \cite{Mahomed}). The principles underlying contact transformations were first formulated by  Lie and  Engel \cite{Lie1,Lie2,Lie3,Lie4,Lie5}.  Lie and  Scheffers \cite{Lie6},  Yumaguzhin \cite{Yumaguzhin1997,Yumaguzhin1996}, and  Wafo et al. \cite{Soh2002} studied the local classification of third-order linear ODEs under contact transformations. They showed that such equations admit at most a ten-dimensional contact symmetry algebra, with equality if and only if they are locally equivalent to $u'''=0$. See also \cite{Svi1,Svi2,Ibragimov1977}. The Laguerre–Forsyth (see \cite{Mahomed1990,Mahomed}) canonical form for scalar linear third-order ODEs is
\begin{equation}\label{1.1}
u''' + a(x)u = 0. \end{equation}
If $a \not\equiv 0$, equation (\ref{1.1}) admits a four- or five-dimensional Lie symmetry algebra. Chern \cite{Chern1940} first applied Cartan’s equivalence method to linearization via contact transformations, deriving conditions for equivalence to (\ref{1.1}) in the cases $a\equiv0$ and $a\equiv1$. Neut and Petitot \cite{Neut2002} later extended these results to arbitrary $a(x)$. They showed that third-order ODEs admitting four- and five-dimensional Lie symmetry algebras under contact transformations are characterized by a nonvanishing W$\ddot{\textrm{u}}$nschmann relative invariant \cite{Neut2002,Wnschmann1905}, together with the vanishing of $14$ additional invariants. Their work applies to the general linear form without addressing the construction of contact transformations. Ibragimov and Meleshko \cite{Ibra} investigated the linearization problem for third-order ODEs using a direct approach via point and contact transformations. They also addressed the construction of transformations to the Laguerre–Forsyth form via both types of transformations. Nonetheless, those transformations were derived as solutions to a nonlinear system of partial differential equations (PDEs). Their work applies to the Laguerre–Forsyth form without distinguishing between canonical forms with four- and five-dimensional Lie symmetry algebras. Al-Dweik et al. \cite{Dweik2019,Dweik2018_2} applied a new Cartan-based framework to the linearization problem for third-order ODEs admitting four- and five-dimensional Lie symmetry algebras via point transformations, yielding invariant coframes and auxiliary functions for constructing transformations by solving systems of linear or Riccati PDEs. We generalize the results in \cite{Dweik2019,Dweik2018_2} to analyze the equivalence problem for linearizable third-order ODEs admitting four- and five-dimensional Lie symmetry algebras under contact transformations. It is noteworthy that the derived theorems are applicable to both proper contact and point transformations. Moreover, we clearly distinguish between linear canonical forms admitting four and those admitting five Lie symmetries. We also provide a method for constructing contact transformations that reduce the equation to its linear canonical form, based on invariant coframes obtained via Cartan’s equivalence method, by solving systems of linear or Riccati PDEs. It should be mentioned that the canonical forms of linearizable third-order ODEs admitting four- and five-dimensional Lie point symmetry algebras in \cite{Dweik2019,Dweik2018_2} are not equivalent under contact transformations. Therefore, we consider the same canonical forms to study linearization problem for third-order ODEs with four- and five-dimensional Lie symmetry algebras under contact transformations, which are summarized in the following Table. For the remainder of this discussion, we adopt the notation $u'=p$ and $u''=q$.
\begin{table}[H]
\centering
\renewcommand{\arraystretch}{1.6}
\begin{tabular}{|l|l|l|}
\hline
\textbf{Algebra} & \textbf{Point symmetry realizations} & \textbf{Representative equations} \\  \hline\hline

$L_{4;1}$&$X_1=\partial_u, X_2=u\partial_u, X_3=f_1(x)\partial_u, X_4=f_2(x)\partial_u$&$u'''=a^3(x)u$, \\
&$f_i'''+a(x)f_i'=0,\,i=1,2$&$a$ not const,~$ \left(\frac{2aa''-3a'^2}{a^4}\right)_x\neq 0$ \\ \hline

$L_{5;1}$&$X_1=\partial_x, X_2=u\partial_u, X_3=f_1(x)\partial_u, X_4=f_2(x)\partial_u,$&$u'''=sp+u$, \\
& $X_5=f_3(x)\partial_u$,$f_i'''+a_1f_i'+a_0f_i=0$,$i=1,\ldots,3$& $s$ constant\\ \hline
\end{tabular}
\caption{Canonical forms of linearizable third-order ODEs under contact transformations admitting  four- and five-dimensional Lie  symmetry algebras}
\end{table}
The paper is organized as follows. Section 2 applies Cartan’s equivalence method to third-order ODEs under contact transformations, yielding invariant coframes for two branches that characterize the canonical forms given in Table 1. Section 3 presents the Principal theorems. Section 4 describes a procedure for constructing the corresponding contact transformations based on these invariant coframes, with illustrative examples. A brief conclusion is given at the end.
\section{Implementation of Cartan’s equivalence method}
\setcounter{equation}{0}
Definitions, preliminaries, notation, and key results relevant to this section can be found in \cite{Olver1995,Neut2003}.
Consider $\left(x, u, p = u', q = u''\right) \in \mathbb{R}^4$ as local coordinates for the second-order jet space $\mathbf{J}^2$.  Throught this paper, the 1-forms $\pi^{\prime\kappa},~\kappa=1, \ldots, 9$,   denote the modified Maurer-Cartan forms.  On the manifold $M=J^2$, we introduce the following base coframe
\begin{equation}\label{2.1}
\left(\begin{array}{l}
\omega^1\\\omega^2\\\omega^3\\\omega^4\\
\end{array}\right)
=\left(\begin{array}{c}
du-pdx\\
dp-qdx\\
dq-fdx\\
dx\\
\end{array}\right),
\end{equation}
In this paper, we consider the equivalence  of
\begin{equation}\label{2.3}
\begin{array}{l}
u^{\prime \prime \prime}=f\left(x, u, u^{\prime},u^{\prime \prime} \right), \quad \bar{u}^{\prime \prime\prime}=\bar{f}\left(\bar{x}, \bar{u}, \bar{u}^{\prime},\bar{u}^{\prime \prime}\right),
\end{array}
\end{equation}
under a contact transformation
\begin{equation}\label{2.4}
\bar{x}=\varphi \left( x,u,p \right),~\bar{u} =\psi \left( x,u,p  \right), ~\bar{p}=\chi(x,u,p),\\
\end{equation}
with the contact condition $d\bar{u}-\bar{p}~d\bar{x}=\lambda(du-p~dx)$ for some function $\lambda(x,u,p)$ and non-zero Jacobian.
This equivalence problem can be expressed, with respect to the base coframe (\ref{2.1}), as the following fundamental equivalence conditions.
\begin{equation}\label{2.5}
\Phi^*\left(\begin{array}{c}
\bar{\omega}^1 \\
\bar{\omega}^2 \\
\bar{\omega}^3\\
\bar{\omega}^4
\end{array}\right)=\left(\begin{array}{cccc}
a_1 & 0 & 0&0 \\
a_2 & a_3 & 0 &0\\
a_4 & a_5 & a_6&0\\
a_7&a_8&0&a_9
\end{array}\right)\left(\begin{array}{c}
\omega^1 \\
\omega^2 \\
\omega^3\\
\omega^4
\end{array}\right),
\end{equation}
for functions $a_i(x,u,p,q),~ i=1,...,9$,  where $\Phi^*$ is the pullback arising from the first prolongation of the contact transformation (\ref{2.4}). Consequently, the corresponding structure group is a nine-dimensional Lie group
\begin{equation}\label{2.5}
G=\left\{
\left(\begin{array}{cccc}
a_1 & 0 & 0&0 \\
a_2 & a_3 & 0 &0\\
a_4 & a_5 & a_6&0\\
a_7&a_8&0&a_9
\end{array}\right)  \Bigg \vert a_1a_3a_6a_9\neq 0
\right\}.
\end{equation}
Let $\theta$ be the lifted coframe defined by
\begin{equation}\label{2.6}
\begin{array}{ll}
\left(\begin{array}{c}
\theta^1, ~\theta^2,~\theta^3,~\theta^4
\end{array}\right)^T=S \left(\begin{array}{c}
\omega^1,~
\omega^2,~
\omega^3,~
\omega^4
\end{array}\right)^T,
\end{array}
\end{equation}
where $S\in G$. 
After completing the absorption step, the first structure equation for (\ref{2.6}) can be written as
\begin{equation}\label{2.7}
\begin{array}{ll}
d\left(\begin{array}{c}
\theta^1 \\
\theta^2 \\
\theta^3\\
\theta^4
\end{array}\right)=\left(\begin{array}{cccc}
\pi^{\prime1} & 0 & 0&0 \\
\pi^{\prime2}&  \pi^{\prime3} & 0&0 \\
\pi^{\prime4}& \pi^{\prime5} & \pi^{\prime6}&0\\
\pi^{\prime7}&\pi^{\prime8}&0&\pi^{\prime9}
\end{array}\right) \wedge\left(\begin{array}{c}
\theta^1 \\
\theta^2 \\
\theta^3\\
\theta^4
\end{array}\right)+\left(\begin{array}{c}
T_{24}^1~ \theta^2 \wedge \theta^4 \\
T_{34}^2~\theta^3\wedge\theta^4 \\
0\\
0
\end{array}\right),
\end{array}
\end{equation}
The essential torsion coefficients can be written explicitly as $T_{24}^1=-\frac{a_1}{a_3 a_9}$ and $T_{34}^2=-\frac{a_3}{a_6 a_9}$
which can be normalized to $-1$ via an appropriate choice of the group parameters 
$a_6 = \frac{a_3^2}{a_1}$ and $a_9 = \frac{a_1}{a_3}$, leading to a reduction of the structure group 
\begin{equation}\label{2.8}
{G}_1=\left\{
\left(\begin{array}{cccc}
a_1 & 0 & 0&0 \\
a_2 & a_3 & 0 &0\\
a_4 & a_5 & \frac{a^2_3} {a_1}&0\\
a_7&a_8&0&\frac{a_1}{a_3}
\end{array}\right)
 \Bigg \vert a_1a_3\neq 0
\right\},
\end{equation}
which  yields the adapted coframe (\ref{2.6}) with $S\in G_1$. \\

In the \emph{second iteration} of the reduction process, after absorption, the structure equations take the following form
\begin{equation}\label{2.10}
d\left(\begin{array}{c}
\theta^1 \\
\theta^2 \\
\theta^3\\
\theta^4
\end{array}\right)=\left(\begin{array}{cccc}
 \pi^{\prime1} & 0 & 0&0 \\
\pi^{\prime2} &  \pi^{\prime3} & 0&0 \\
\pi^{\prime4} &\pi^{\prime5}& 2\pi^{\prime3}-\pi^{\prime1}&0\\
\pi^{\prime6}&\pi^{\prime7}&0&\pi^{\prime1}-\pi^{\prime3}
\end{array}\right) \wedge\left(\begin{array}{c}
\theta^1 \\
\theta^2 \\
\theta^3\\
\theta^4
\end{array}\right)+\left(\begin{array}{c}
- \theta^2 \wedge \theta^4 \\
 - \theta^3 \wedge \theta^4 \\
T_{34}^3~\theta^3 \wedge \theta^4 \\
0
\end{array}\right).
\end{equation}
The essential torsion coefficient can be written explicitly as  $T_{34}^3=\frac{a_3^2I_1-3a_1a_5+3a_2a_3}{a_1a_3}$
which can be reduced to zero by an appropriate normalization of the group parameter 
  $a_5=\frac{a_2a_3}{a_1}+\frac{a_3^2}{3a_1}I_1$, where $I_1=-f_q$. 
Accordingly, the structure group takes the following reduced form
\begin{equation}\label{2.8}
{G}_2=\left\{
\left(\begin{array}{cccc}
a_1 & 0 & 0&0 \\
a_2 & a_3 & 0 &0\\
a_4 &\frac{a_2a_3}{a_1}+\frac{a_3^2 }{3a_1}I_1& \frac{a^2_3} {a_1}&0\\
a_7&a_8&0&\frac{a_1}{a_3}
\end{array}\right)
 \Bigg \vert a_1a_3\neq 0
\right\},
\end{equation}
which  yields the adapted coframe (\ref{2.6}) with $S\in G_2$. \\

In the third iteration of the reduction process, after absorption, the structure equations are expressed in the following form
\begin{equation}\label{2.7}
\begin{array}{ll}
d\left(\begin{array}{c}
\theta^1 \\
\theta^2 \\
\theta^3\\
\theta^4
\end{array}\right)=\left(\begin{array}{cccc}
\pi^{\prime1} & 0 & 0&0 \\
\pi^{\prime2}&  \pi^{\prime3} & 0&0 \\
 \pi^{\prime4}& \pi^{\prime2} & 2\pi^{\prime3}-\pi^{\prime1}&0\\
 \pi^{\prime5}&\pi^{\prime6}&0&\pi^{\prime1}-\pi^{\prime3}
\end{array}\right) \wedge\left(\begin{array}{c}
\theta^1 \\
\theta^2 \\
\theta^3\\
\theta^4
\end{array}\right)+\left(\begin{array}{c}
-\theta^2 \wedge \theta^4 \\
-\theta^3\wedge\theta^4 \\
T_{24}^3~\theta^2 \wedge \theta^4\\
0
\end{array}\right).
\end{array}
\end{equation}
The essential torsion coefficient can be written explicitly as  $T_{24}^3=\frac{a_3^2I_2-2a_1a_4+a_2^2}{a_1^2}$, 
which can be reduced to zero by an appropriate normalization of the group parameter  $a_4=\frac{a_2^2}{2a_1}+\frac{a_3^2}{2a_1}I_2$, where  $I_2=-\frac{2}{9}I_1^2-f_p-\frac{1}{3}\hat{D}_xI_1$. 
Consequently, the structure group is reduced  to
\begin{equation}\label{2.8}
{G}_3=\left\{
\left(\begin{array}{cccc}
a_1 & 0 & 0&0 \\
a_2 & a_3 & 0 &0\\
\frac{a_2^2}{2a_1}+\frac{a_3^2}{2a_1}I_2 &\frac{a_2a_3}{a_1}+\frac{a_3^2}{3a_1}I_1 & \frac{a^2_3} {a_1}&0\\
a_7&a_8&0&\frac{a_1}{a_3}
\end{array}\right)
 \Bigg \vert a_1a_3\neq 0
\right\},
\end{equation}
which  yields the adapted coframe (\ref{2.6}) with $S\in G_3$. \\

Proceeding to the \emph{fourth iteration} of the reduction algorithm, one finds that, after absorption, the structure equations become
\begin{equation}\label{2.7}
\begin{array}{ll}
d\left(\begin{array}{c}
\theta^1 \\
\theta^2 \\
\theta^3\\
\theta^4
\end{array}\right)=\left(\begin{array}{cccc}
\pi^{\prime1} & 0 & 0&0 \\
\pi^{\prime2}&  \pi^{\prime3} & 0&0 \\
0& \pi^{\prime2} & 2\pi^{\prime3}-\pi^{\prime1}&0\\
\pi^{\prime4}&\pi^{\prime5}&0&\pi^{\prime1}-\pi^{\prime3}
\end{array}\right) \wedge\left(\begin{array}{c}
\theta^1 \\
\theta^2 \\
\theta^3\\
\theta^4
\end{array}\right)+\left(\begin{array}{c}
-\theta^2 \wedge \theta^4 \\
-\theta^3\wedge\theta^4 \\
T_{14}^3~\theta^1 \wedge \theta^4\\
0
\end{array}\right).
\end{array}
\end{equation}
The essential  torsion coefficient is
\begin{equation}
T^3_{14}=\frac{a_3^3}{a_1^3}I_3,\\
\end{equation}
where 
\begin{equation}
I_3=-\frac{1}{3}I_1I_2-f_u-\frac{1}{2}\hat{D}_xI_2. 
\end{equation}
As a result, we obtain the following branch.
\subsection{Branch $I_{3}\ne 0$ :}
The group parameter $a_{3}$ an be  normalized as
\begin{equation}
a_{3}=\frac{a_1}{J_3},
\end{equation}
 by translating $T^3_{14}$ to one, where $I_3=J_3^3$.\\
 
Proceeding with another iteration of reduction and absorption leads to the following structure equations
\begin{equation}\label{2.7}
\begin{array}{ll}
d\left(\begin{array}{c}
\theta^1 \\
\theta^2 \\
\theta^3\\
\theta^4
\end{array}\right)=\left(\begin{array}{cccc}
\pi^{\prime1} & 0 & 0&0 \\
\pi^{\prime2}& \pi^{\prime1}& 0&0 \\
0& \pi^{\prime2} &\pi^{\prime1}&0\\
\pi^{\prime3}&\pi^{\prime4}&0&0
\end{array}\right) \wedge\left(\begin{array}{c}
\theta^1 \\
\theta^2 \\
\theta^3\\
\theta^4
\end{array}\right)
+\left(\begin{array}{c}
-\theta^2 \wedge \theta^4 \\
T^2_{23}~\theta^2 \wedge \theta^3+T^2_{24}~\theta^2 \wedge \theta^4-\theta^3 \wedge \theta^4\\
\theta^1 \wedge \theta^4+T^3_{23}~\theta^2 \wedge \theta^3+T^3_{34}~\theta^3 \wedge \theta^4\\
T^4_{34}~\theta^3 \wedge \theta^4
\end{array}\right),
\end{array}
\end{equation}
and the essential torsion coefficients are given by
\begin{equation}
\begin{array}{lll}
T^{2}_{23}=\frac{J_3(I_4-a_8)}{a_1},
T^2_{24}=\frac{a_1I_5-a_2}{a_1},
T^3_{23}=\frac{1}{a_1}I_6~ mod~T^2_{23},
\end{array}
\end{equation}
where
\begin{equation}
\begin{array}{lll}
I_4=J_{3_q},
I_5=\frac{1}{3J_3^2}(I_1J_3+3\hat{D}_xJ_3),
I_6=\frac{2}{3}J_{3_q}I_1-\frac{1}{3}I_{1_q}J_3-2J_{3_p}+2J_3I_4I_5.
\end{array}
\end{equation}
We normalize the essential torsion coefficients 
$T^2_{23}=0, T^2_{24}=0$  by setting 
\begin{equation}
\begin{array}{lll}
 a_{2}=a_1I_5,& a_8=I_{4}.
\end{array}
\end{equation}
After these normalizations, it is noted that $I_6$ is a relative invariant and $L_{4;1}, L_{5;1} $ belong to the sub-branch $I_6=0$. 
Consequently, the structure group is reduced to
\begin{equation}
{G}_4=\left\{\left(\begin{array}{cccc}
a_1 & 0 & 0&0 \\
a_1I_5 & \frac{a_1}{J_3}& 0 &0\\
(\frac{I_5^2}{2}+\frac{I_2}{2J_3^2})a_1&(\frac{I_5}{J_3}+\frac{I_1}{3J_3^2})a_1 & \frac{a_1} {J_3^2}&0\\
a_7&I_4&0&{J_3}
\end{array}\right)
 \Bigg \vert a_1\neq 0
\right\},
\end{equation}
which  yields the adapted coframe (\ref{2.6}) with $S\in G_4$. \\

Continuing with a further cycle of reduction and absorption produces the following structure equations
\begin{equation}
\begin{array}{ll}
d\left(\begin{array}{c}
\theta^1 \\
\theta^2 \\
\theta^3\\
\theta^4
\end{array}\right)=\left(\begin{array}{cccc}
\pi^{\prime1} & 0 & 0&0 \\
0& \pi^{\prime1}& 0&0 \\
0& 0 &\pi^{\prime1}&0\\
\pi^{\prime2}&0&0&0
\end{array}\right) \wedge\left(\begin{array}{c}
\theta^1 \\
\theta^2 \\
\theta^3\\
\theta^4
\end{array}\right)\\
+\left(\begin{array}{c}
-\theta^2 \wedge \theta^4 \\
T^2_{13}~\theta^1\wedge \theta^3+T^2_{14}~\theta^1 \wedge \theta^4-\theta^3 \wedge \theta^4\\
T^3_{12}~\theta^1 \wedge \theta^2+T^3_{23}~\theta^2 \wedge \theta^3+T^3_{24}~\theta^2 \wedge \theta^4+\theta^1 \wedge \theta^4\\
T^4_{23}~\theta^2 \wedge \theta^3+T^4_{24}~\theta^2 \wedge \theta^4
\end{array}\right),
\end{array}
\end{equation}
 and the essential torsion coefficientes
\begin{equation}
\begin{array}{lll}
T^{2}_{13}=\frac{I_7-a_7}{a_1},\,\
T^4_{23}=-\frac{J_3^3}{a_1^2}I_8,\,\
\end{array}
\end{equation}
where
\begin{equation}
\begin{array}{ll}
 I_7=-I_{5_q}J_3^2,\,\ I_8=I_{4_q}.
\end{array}
\end{equation}
It is noted that $I_8$ is a relative invariant and $L_{4;1}, L_{5;1} $ belong to the sub-branch $I_8=0$. Normalizing the essential torsion coefficient $ T^2_{13}=0$ by setting
\begin{equation}
a_7=I_7.
\end{equation}
Therefore, the structure group is reduced to
\begin{equation}\label{2.8}
{G}_5=\left\{
\left(\begin{array}{cccc}
a_1 & 0 & 0&0 \\
a_1I_5 & \frac{a_1}{J_3}& 0 &0\\
(\frac{I_5^2}{2}+\frac{I_2}{2J_3^2})a_1&(\frac{I_5}{J_3}+\frac{I_1}{3J_3^2})a_1 & \frac{a_1} {J_3^2}&0\\
I_7&I_4&0&{J_3}
\end{array}\right)
 \Bigg \vert a_1\neq 0
\right\},
\end{equation}
which yields the adapted coframe (\ref{2.6}) with $S\in G_5$. \\

Continuing with a further cycle of reduction and absorption produces the following structure equations
\begin{equation}
\begin{array}{ll}
d\left(\begin{array}{c}
\theta^1 \\
\theta^2 \\
\theta^3\\
\theta^4
\end{array}\right)=\left(\begin{array}{cccc}
\pi^{\prime1} & 0 & 0&0 \\
0& \pi^{\prime1}& 0&0 \\
0& 0 &\pi^{\prime1}&0\\
0&0&0&0
\end{array}\right) \wedge\left(\begin{array}{c}
\theta^1 \\
\theta^2 \\
\theta^3\\
\theta^4
\end{array}\right)
+\left(\begin{array}{c}
-\theta^2 \wedge \theta^4 \\
T^2_{14}~\theta^1\wedge \theta^4-\theta^3 \wedge \theta^4\\
T^3_{12}~\theta^1 \wedge \theta^2+\theta^1 \wedge \theta^4+T^3_{24}~\theta^2 \wedge \theta^4\\
T^4_{12}~\theta^1 \wedge \theta^2+T^4_{14}~\theta^1\wedge \theta^4
\end{array}\right),
\end{array}
\end{equation}
and the essential torsion coefficientes 
\begin{equation}\label{2.33}
\begin{array}{lll}
T^{2}_{14}=-\frac{1}{2J_3^2}I_{9},\,\ T^4_{12}=\frac{1}{a_1^2}I_{10},\,\
T^4_{14}=\frac{1}{a_1J_3}I_{11}, \,\ T^3_{12}=\frac{1}{18a_1J_3^2}I_{12},
\end{array}
\end{equation}
where
\begin{equation}
\begin{array}{lll}
I_{9}=2J_3\hat{D}_xI_5-I_2+J_3^2I_5^2,\,\ 
I_{10}=I_4\hat{D}_xI_7-J_3I_{7_p}+J_3^2\left(\frac{I_4}{J_3}\right)_{u},\,\ 
I_{11}=J_{3_u}-\hat{D}_xI_7,\\
I_{12}=6I_1J_3(I_5I_7-I_4\hat{D}_xI_5)-18J_3^3I_4+6I_{5_p}I_1J_3^2-18I_2J_3I_4I_5-18J_3I_7\hat{D}_xI_5\\
+18J_3^2I_{5_u}+18I_2\hat{D}_xI_4+2I_1^2I_7-9I_{2_p}J_3+6J_3^2\left(\frac{I_1}{J_3}\right)_{u}
+36I_2I_7.
\end{array}
\end{equation}
It is noted that $I_{10}, I_{11}, I_{12}$ are relative invariants and $L_{4;1}, L_{5;1}$ belong to the 
sub-branch $I_{10}= I_{11}= I_{12}=0$.
In this sub-branch, all unabsorbable torsions vanish, and therefore the remaining group parameter $a_1$ cannot be normalized. Moreover, the 1-form $\pi^{\prime 1}$ is uniquely determined, which implies that the equivalence problem is completely resolved. Consequently, this results in the following $e$-structure on the five-dimensional prolonged manifold $M^{(1)} = M \times G_5$.
\begin{equation}\label{2.35}
\left(\begin{array}{l}
\theta^1 \\
\theta^2 \\
\theta^3 \\
\theta^4\\
\pi^{\prime 1}
\end{array}\right)=\left(\begin{array}{ccccc}
a_1 &0 & 0&0&0 \\
a_1I_5 &\frac{a_1}{J_3}  & 0 &0&0\\
(\frac{I_5^2}{2}+\frac{I_2}{2J_3^2})a_1&(\frac{I_5}{J_3}+\frac{I_1}{3J_3^2})a_1  & \frac{a_1} {J_3^2}&0&0\\
I_7&I_4&0&J_3&0\\
-\frac{I_1I_7+3J_3^2I_{5_p}+3J_{3_u}-3J_3I_4\hat{D}_xI_5}{3J_3}&\frac{I_7}{J_3}-I_4I_5&0&-I_5J_3&\frac{1}{a_1}
\end{array}\right)\left(\begin{array}{c}
\omega^1 \\
\omega^2 \\
\omega^3\\
\omega^4\\
da_1
\end{array}\right) .
\end{equation}
This yields the following structure equations
\begin{equation}
\begin{split}
&d\theta^{1}=\pi^{\prime1}\wedge\theta^{1}-\theta^2\wedge\theta^4,\\
&d\theta^{2}=\pi^{\prime1}\wedge\theta^{2}-\frac{K}{2}\theta^{1}\wedge\theta^{4}-\theta^{3}\wedge\theta^{4},\\
&d\theta^{3}=\pi^{\prime1}\wedge\theta^3+\theta^{1}\wedge\theta^{4}-\frac{K}{2}\theta^{2}\wedge\theta^{4},\\
&d\theta^{4}=0,\\
&d\pi^{\prime 1}=0,
\end{split}
\end{equation}
where $K=\frac{I_9}{J_3^2}$. The values of $K$ provides us  with the following two branches.
\subsubsection{ Branch $K_x=K_u=K_p=K_q=0$ ($K$ is constant)}
It should be noted here that the invariant $K=s$ for the canonical form \\$u'''=su'+u.$
As a consequence of Cartan’s equivalence method, we obtain an invariant coframe \eqref{2.35} of rank zero and order zero on the
five-dimensional prolonged space with local coordinates $(x,u,p,q,a_1)$. This coframe defines an $e$-structure whose
structure functions are complete contact invariants of the underlying third-order differential equation.  By applying \cite[Theorem 8.22, page 275]{Olver1995}, every equation belonging to this branch admits a five-dimensional Lie group of contact transformations.\\

By applying \cite[Theorem 8.19, page 271]{Olver1995}, we have proved Theorem \ref{Th3.1}, which will be stated in the next section.
\subsubsection{  Branch $\hat{D}_xK\ne 0$ ($K$ is not a constant)} 
The invariant derivations are given by the vector fields dual to the lifted coframe \eqref{2.35}.
\begin{equation}
\begin{aligned}
\frac{\partial}{\partial\theta^1}&=\frac{J_3I_4I_5-I_7}{J_3a_1}\hat{D}_x+\frac{1}{a_1}\frac{\partial}{\partial u}-\frac{J_3I_5}{a_1}\frac{\partial}{\partial p}+\frac{1}{6a_1J_3}(3J_3^3I_5^2+2J_3^2I_1I_5-3I_2J_3)\frac{\partial}{\partial q}\\
&+\frac{1}{3J_3}(-3J_3I_4\hat{D}_xI_5+3I_{5_p}J_3^2+I_1I_7+3J_{3_u})\frac{\partial}{\partial a_1},\\
\frac{\partial}{\partial\theta^2}&=-\frac{I_4}{a_1}\hat{D}_x+\frac{J_3}{a_1}\frac{\partial}{\partial p}-\frac{1}{3a_1}(3J_3^2I_5+J_3I_1)\frac{\partial}{\partial q}-I_7\frac{\partial}{\partial a_1},\\
\frac{\partial}{\partial\theta^3}&=\frac{J_3^2}{a_1}\frac{\partial}{\partial q},\\
\frac{\partial}{\partial\theta^4}&=\frac{1}{J_3}\hat{D}_x+I_5a_1\frac{\partial}{\partial a_1},\\
\frac{\partial}{\partial\pi^{\prime 1}}&=a_1\frac{\partial}{\partial a_1}.
\end{aligned}
\end{equation}
It is important to observe that the invariants
\begin{equation}
\begin{array}{lll}
\frac{\partial K}{\partial\theta^1}=\frac{\partial K}{\partial\theta^2}=\frac{\partial K}{\partial\theta^3}=\frac{\partial K}{\partial\pi^{\prime 1}}=0, \frac{\partial K}{\partial\theta^4}\ne 0,
\end{array}
\end{equation}
hold for the canonical form $u'''=a^3(x)u$. Thus we choose the sub-branch $\frac{\partial K}{\partial\theta^1}=0, \frac{\partial K}{\partial\theta^2}=0, \frac{\partial K}{\partial\theta^3}=0, \frac{\partial K}{\partial\pi^{\prime 1}}=0$, which is equivalents to the sub-branch
\begin{equation}
\begin{array}{lll}
 K_q=0,\quad (J_3I_4I_5-I_7)\hat{D}_xK+J_3K_u-J_3^2I_5K_p=0, \\
I_4\hat{D}_xK-J_3K_p=0.
\end{array}
\end{equation}
Moreover, the first-order classifying set is
\begin{equation}
C^{(1)}=\left\{K, \frac{\partial K}{\partial\theta^4}\right\}.
\end{equation}
The invariants $K$ and $\frac{\partial K}{\partial\theta^4}$ are functionally dependent. As a consequence of Cartan’s equivalence method, we obtain an invariant coframe \eqref{2.35} of rank one and order zero on the
five-dimensional prolonged space with local coordinates $(x,u,p,q,a_1)$. This coframe defines an $e$-structure whose
structure functions are complete contact invariants of the underlying third-order ODE.  By applying \cite[Theorem 8.22, page 275]{Olver1995}, every equation belonging to this branch admits a four-dimensional Lie group of contact transformations. 
By applying \cite[Theorem 8.19, page 271]{Olver1995}, the overlap of the classifying sets $C^{(1)}(\theta)$ and $C^{(1)}(\bar{\theta})$ ensures that the two
$e$-structures are locally equivalent, thereby yielding the necessary and sufficient invariant conditions for contact
equivalence. We conclude that
\begin{equation}\label{2.46}
\begin{array}{lll}
K=\bar{K}=\frac{2\bar{a}(\bar{x})\bar{a}''(\bar{x})-3\bar{a}'(\bar{x})^2}{\bar{a}(\bar{x})^4},\,
\frac{1}{J_3}\hat{D}_xK=-\frac{1}{\bar{a}(\bar{x})}\hat{D}_{\bar{x}}\bar{K}.
\end{array}
\end{equation}
 So, we have proved Theorem \ref{Th3.2}, which will be stated in the next section.
\section{Main Theorems}
\setcounter{equation}{0}
In this section, we present the main theorems established in the previous section.
\begin{theorem} \label{Th3.1}
A scalar third-order ODE  $u''' = f(x,u,u',u'')$ is equivalent to the canonical form
\begin{equation}
\bar{u}'''=s\bar{u}'+\bar{u}, s=constant ,
\end{equation}
with five  symmetries via contact transformations (\ref{2.4}) if and only if  the relative invariants
\begin{equation}\label{3.1}
\begin{array}{lll}
I_6=\frac{2}{3}J_{3_q}I_1-\frac{1}{3}I_{1_q}J_3-2J_{3_p}+2J_3I_4I_5,\\
I_8=I_{4_q},\\
I_{10}=I_4\hat{D}_xI_7-J_3I_{7_p}+J_3^2\left(\frac{I_4}{J_3}\right)_{u},\\
I_{11}=J_{3_u}-\hat{D}_xI_7,\\
I_{12}=6I_1J_3(I_5I_7-I_4\hat{D}_xI_5)-18J_3^3I_4+6I_{5_p}I_1J_3^2-18I_2J_3I_4I_5-18J_3I_7\hat{D}_xI_5\\
+18J_3^2I_{5_u}+18I_2\hat{D}_xI_4+2I_1^2I_7-9I_{2_p}J_3+6J_3^2\left(\frac{I_1}{J_3}\right)_{u}
+36I_2I_7,\\
K_x, K_u, K_p, K_q,
\end{array}
\end{equation}
vanish identically, where
\begin{equation}\label{3.2}
\begin{array}{lll}
I_1=-f_q,&
I_2=-\frac{2}{9}I_1^2-f_p-\frac{1}{3}\hat{D}_xI_1,\\
I_3=J_3^3=-\frac{1}{3}I_1I_2-f_u-\frac{1}{2}\hat{D}_xI_2\ne 0,&
I_4=J_{3_q},\\
I_5=\frac{1}{3J_3^2}(I_1J_3+3\hat{D}_xJ_3),&
I_7=-I_{5_q}J_3^2,\\
I_{9}=2J_3\hat{D}_xI_5-I_2+J_3^2I_5^2,&
K=\frac{I_{9}}{J_3^2}.
\end{array}
\end{equation}
Finally, the constant $s$ appearing in the resulting canonical form is determined by the relation $s = K$.
\end{theorem}
\begin{theorem} \label{Th3.2}
A scalar third-order ODE  $u''' = f(x,u,u',u'')$  is equivalent to  the canonical form
\begin{equation}
 \bar{u}'''=\bar{a}^3(x) \bar{u},  \left(\frac{2\bar{a}\bar{a}''-3\bar{a}'^2}{\bar{a}^4}\right)_x\neq 0,
\end{equation}
with four  symmetries via  contact transformations (\ref{2.4}) if and only if  the relative invariants
\begin{equation}\label{3.1}
\begin{array}{lll}
I_6=\frac{2}{3}J_{3_q}I_1-\frac{1}{3}I_{1_q}J_3-2J_{3_p}+2J_3I_4I_5,\\
I_8=I_{4_q},\\
I_{10}=I_4\hat{D}_xI_7-J_3I_{7_p}+J_3^2\left(\frac{I_4}{J_3}\right)_{u},\\
I_{11}=J_{3_u}-\hat{D}_xI_7,\\
I_{12}=6I_1J_3(I_5I_7-I_4\hat{D}_xI_5)-18J_3^3I_4+6I_{5_p}I_1J_3^2
-18I_2J_3I_4I_5
-18J_3I_7\hat{D}_xI_5\\
+18J_3^2I_{5_u}+18I_2\hat{D}_xI_4+2I_1^2I_7-9I_{2_p}J_3+6J_3^2\left(\frac{I_1}{J_3}\right)_{u}
+36I_2I_7,\\
I_{13}=(J_3I_4I_5-I_7)\hat{D}_xK+J_3K_u-J_3^2I_5K_p,\\
I_{14}=K_q,\\
I_{15}=I_4\hat{D}_xK-J_3K_p,
\end{array}
\end{equation}
vanish identically, where
\begin{equation}\label{3.2}
\begin{array}{lll}
I_1=-f_q,&
I_2=-\frac{2}{9}I_1^2-f_p-\frac{1}{3}\hat{D}_xI_1,\\
I_3=J_3^3=-\frac{1}{3}I_1I_2-f_u-\frac{1}{2}\hat{D}_xI_2\ne 0,&
I_4=J_{3_q},\\
I_5=\frac{1}{3J_3^2}(I_1J_3+3\hat{D}_xJ_3),&
I_7=-I_{5_q}J_3^2,\\
I_{9}=2J_3\hat{D}_xI_5-I_2+J_3^2I_5^2,&
K=\frac{I_{9}}{J_3^2},\\
D_xK\ne0.
\end{array}
\end{equation}
\end{theorem}
\section{Construction of contact transformations based on invariant coframes}
\setcounter{equation}{0}
We construct contact transformations between equivalent third-order ODEs in Table 1 using invariant coframes and the following proposition.
\begin{proposition} 
Assume that the third-order ODEs
\begin{equation}\label{e4.8}
\begin{array}{lll}
u^{\prime \prime\prime}=f\left(x, u, u^{\prime},u^{\prime \prime}\right), \quad \bar{u}^{\prime \prime\prime}=s\bar{u}'+\bar{u},
\end{array}
\end{equation}
are  equivalent  under the contact transformation
\begin{equation}\label{e4.9}
\bar x = \varphi (x,u,p),\,\,\bar u = \psi (x,u,p),\,\, \bar{p}=\chi(x,u,p),
\end{equation} 
with non-zero Jacobian. Then the linearizing contact transformation (\ref{e4.9}) between  the third-order ODEs (\ref{e4.8}) can be obtained by the following systematic way :
\begin{description} 
\item[Step 1] Find non-zero solution $a_1$ for the \textbf{first order linear system of PDEs}
\begin{equation}\label{4.3}
\begin{array}{l}
D_xa_1 =J_3I_5a_1,~a_{1_u}=-Qa_1, ~a_{1_p}=(I_4I_5-\frac{I_7}{J_3})a_1,~a_{1_q}=0.\\
\end{array}
\end{equation}
\item[Step 2] Find non-zero solution $\varphi$ for the \textbf{first order linear system of PDEs}
\begin{equation}\label{4.4}
\begin{array}{l}
\hat{D}_x \varphi=-J_3 ,\,\,\varphi _u=-I_7 ,\,\,\varphi_p=-I_4.\\
\end{array}
\end{equation}
\item[Step 3] Find non-zero solutions $\eta, \chi, \psi $ for the \textbf{first order linear system of PDEs}
\begin{equation}\label{4.5}
\begin{array}{l}
\hat{D}_x\eta=-J_3\bar{f} ,\,\, \eta_u=(\frac{s+I_5^2}{2}+\frac{I_2}{2J_3^2})a_1-I_7\bar{f} , \,\,\eta_p=(\frac{I_5}{J_3}+\frac{I_1}{3J_3^2})a_1-I_4\bar{f},\,\,\eta_q=\frac{a_1}{J_3^2}.\\
\hat{D}_x\chi=-J_3~\eta, \,\, \chi_u=-I_5a_1-I_7~\eta,\,\,\chi_p=-\frac{a_1}{J_3}-I_4~\eta.\\
{I_7}\hat{D}_x\psi=J_3(\psi_u-a_1),\,\,\psi_x=\chi\varphi_x-pa_1,\,\,\psi_p=-I_4~\chi.
\end{array}
\end{equation}
\end{description} 
where $a_1(x,u,p)$ is auxiliary function, $\chi=\frac{\hat{D}_x\psi}{\hat{D}_x\varphi },\,\, \eta=\frac{\hat{D}_x\chi}{\hat{D}_x\varphi },\,\,\bar{f}=s{\chi}+{\psi},$ and
\begin{equation*}
\begin{array}{lll}
I_1=-f_q,&
I_2=-\frac{2}{9}I_1^2-f_p-\frac{1}{3}\hat{D}_xI_1,\\
I_3=J_3^3=-\frac{1}{3}I_1I_2-f_u-\frac{1}{2}\hat{D}_xI_2,&
I_5=\frac{1}{3J_3^2}(I_1J_3+3\hat{D}_xJ_3),\\
 I_7=-I_{5_q}J_3^2,&
I_{9}=2J_3\hat{D}_xI_5-I_2+J_3^2I_5^2,\\
Q=-\frac{I_1I_7+3J_3^2I_{5_p}+3J_{3_u}-3J_3I_4\hat{D}_xI_5}{3J_3},&
s=K=\frac{I_9}{J_3^2}.
\end{array}
\end{equation*}
\proof
The equivalence of the third-order  ODEs (\ref{e4.8})  under the contact transformation (\ref{e4.9}) can be checked using the  invariant five-dimensional coframe (\ref{2.35}) on the space $M^{(1)}=M\times G_5$ such that
\begin{equation}\label{ee4.10}
\begin{aligned}
&\Phi^*\left(\begin{array}{ccccc}
\bar{a}_1 &0 & 0&0&0 \\
\bar{a}_1\bar{I}_5 &\frac{\bar{a}_1}{\bar{J}_3}  & 0 &0&0\\
(\frac{\bar{I}_5^2}{2}+\frac{\bar{I}_2}{2\bar{J}_3^2})\bar{a}_1&(\frac{\bar{I}_5}{\bar{J}_3}+\frac{\bar{I}_1}{3\bar{J}_3^2})\bar{a}_1  & \frac{\bar{a}_1} {\bar{J}_3^2}&0&0\\
\bar{I}_7&\bar{I}_4&0&\bar{J}_3&0\\
-\frac{\bar{I}_1\bar{I}_7+3\bar{J}_3^2\bar{I}_{5_{\bar{p}}}+3\bar{J}_{3_{\bar{u}}}-3\bar{J}_3I_4\hat{D}_{\bar{x}}\bar{I}_5}{3\bar{J}_3}&\frac{\bar{I}_7}{\bar{J}_3}-\bar{I}_4\bar{I}_5&0&-\bar{I}_5\bar{J}_3&\frac{1}{\bar{a}_1}
\end{array}\right) \left(\begin{array}{c}
 \bar{\omega}^1\\\bar{\omega}^2\\\bar{\omega}^3\\\bar{\omega}^4\\
d\bar{a}_1\\
\end{array}\right)=\\
&\left(\begin{array}{ccccc}
a_1 &0 & 0&0&0 \\
a_1I_5 &\frac{a_1}{J_3}  & 0 &0&0\\
(\frac{I_5^2}{2}+\frac{I_2}{2J_3^2})a_1&(\frac{I_5}{J_3}+\frac{I_1}{3J_3^2})a_1  & \frac{a_1} {J_3^2}&0&0\\
I_7&I_4&0&J_3&0\\
-\frac{I_1I_7+3J_3^2I_{5_p}+3J_{3_u}-3J_3I_4\hat{D}_xI_5}{3J_3}&\frac{I_7}{J_3}-I_4I_5&0&-I_5J_3&\frac{1}{a_1}
\end{array}\right)\left(\begin{array}{c}
\omega^1\\\omega^2\\\omega^3\\\omega^4 \\
da_1\\
\end{array}\right),
\end{aligned}
\end{equation}
where $\Phi^*$  is the pullback arising from the first prolongation of the contact transformation (\ref{e4.9}), and the constant $s$ can be evaluted by $s=\frac{I_9}{J_3^2}$ as we stated in Theorem  \ref{Th3.1}. By incorporating the values  $$\bar{I_1}=0,~\bar{I_2}=-s,~\bar{J_3}=-1,~\bar{I_4}=0,~\bar{I_5}=0,~\bar{I_7}=0, ~\bar{Q}=0$$ for $\bar{u}^{\prime \prime\prime}=s\bar{u}'+\bar{u}$, and $\bar{a}_1=1$,
 equation (\ref{ee4.10}) can be expressed  as
\begin{equation}\label{ee4.15}
\left( {\begin{array}{*{20}c}
\bar{\omega}^1\\\bar{\omega}^2\\\bar{\omega}^3\\\bar{\omega}^4\\ 0\\
\end{array}} \right)  =\left( {\begin{array}{*{20}c}
  a_1 & 0 & 0 &0&0 \\
  -I_5a_1 & -\frac{a_1}{J_3} & 0&0&0  \\
  (\frac{s+I_5^2}{2}+\frac{I_2}{2J_3^2})a_1 &(\frac{I_5}{J_3}+\frac{I_1}{3J_3^2})a_1& \frac{a_1}{J_3^2}&0&0  \\
{-I_7} &{-I_4}&0&{-J_3}&0\\
Q&\frac{I_7}{J_3}-I_4I_5&0&-J_3I_5&\frac{1}{a_1}
\end{array}} \right) \left( {\begin{array}{*{20}c} 
\omega^1\\\omega^2\\\omega^3\\\omega^4\\
da_1\\
\end{array}} \right).
\end{equation}
Computing the pullback associated with the left-hand side of equation (\ref{ee4.15}) gives systems (\ref{4.3}),  (\ref{4.4}), and (\ref{4.5}).\\
\endproof
\end{proposition} 
\begin{proposition} 
Assume that the third-order ODEs
\begin{equation}
\begin{array}{lll}\label{ee4.8}
u^{\prime \prime\prime}=f\left(x, u, u^{\prime},u^{\prime \prime}\right) ,\quad  \bar{u}^{\prime \prime\prime}=\bar{a}^3(\bar{x}) \bar{u}
\end{array}
\end{equation}
are equivalent  under the contact transformation
\begin{equation}\label{ee4.9}
\bar x = \varphi (x,u,p),\,\,\bar u = \psi (x,u,p),\,\, \bar{p}=\chi(x,u,p),
\end{equation} 
with non-zero Jacobian. 
Then the linearizing contact transformation (\ref{ee4.9}) between  the third-order ODEs (\ref{ee4.8}) can be obtained by the following systematic way :
\begin{description} 
\item[Step 1] Find the auxiliary function $H$ satisfying the \textbf{first order system of PDEs}
\begin{equation}\label{4.10}
\begin{array}{l}
-\frac{2}{J_3}\hat{D}_xH+H^2=K. 
\end{array}
\end{equation}
\item[Step 2] Find the auxiliary function $b$ satisfying the \textbf{first order linear system of PDEs}
\begin{equation}\label{4.11}
\begin{array}{l}
\hat{D}_xb=-(J_3H)b.
\end{array}
\end{equation}
\item[Step 3] Find non-zero solution $a_1$ satisfying the \textbf{first order linear system of PDEs}
\begin{equation}\label{4.12}
\begin{array}{l}
D_xa_1 =J_3(H+I_5)a_1, a_{1_u}=(HI_7-Q)a_1, a_{1_p}=(I_4(H+I_5)-\frac{I_7}{J_3})a_1, a_{1_q}=0.\\
\end{array}
\end{equation}
\item[Step 4] Find non-zero solution $\varphi$ satisfying the \textbf{first order linear system of PDEs}
\begin{equation}\label{4.13}
\begin{array}{l}
\hat{D}_x \varphi=-\frac{J_3}{b} ,\,\,\varphi _u=-\frac{I_7}{b} ,\,\,\varphi_p=-\frac{I_4}{b}.\\
\end{array}
\end{equation}
\item[Step 5] Find non-zero solutions $\eta, \chi, \psi $ satisfying the \textbf{first order linear system of PDEs}
\begin{equation}\label{4.14}
\begin{array}{l}
\hat{D}_x\eta=-\frac{J_3}{b}\bar{f} ,\,\, \eta_u=(\frac{(H+I_5)^2}{2}+\frac{I_2}{2J_3^2})b^2a_1-\frac{I_7}{b}\bar{f} , \,\,\eta_p=(\frac{H+I_5}{J_3}+\frac{I_1}{3J_3^2})b^2a_1-\frac{I_4}{b}\bar{f},\,\,\eta_q=\frac{b^2a_1}{J_3^2}.\\
\hat{D}_x\chi=-\frac{J_3}{b}~\eta, \,\, \chi_u=-(H+I_5)ba_1-\frac{I_7}{b}~\eta,\,\,\chi_p=-\frac{ba_1}{J_3}-\frac{I_4}{b}~\eta.\\
{I_7}\hat{D}_x\psi=J_3(\psi_u-a_1),\,\,\psi_x=\chi\varphi_x-pa_1,\,\,\psi_p=-\frac{I_4}{b}~\chi.
\end{array}
\end{equation}
\end{description} 
where $a_1(x,u,p), H(x,u,p), b(x,u,p)$ are auxiliary functions, $\chi=\frac{\hat{D}_x\psi}{\hat{D}_x\varphi },\,\,\\
 \eta=\frac{\hat{D}_x\chi}{\hat{D}_x\varphi },\,\,\bar{f}=\bar{a}^3(\phi){\psi}$, $ H(x,u,p)=\frac{\hat{D}_{\bar{x}}\bar{a}(\bar{x})}{\bar{a}^2(\bar{x})}$, $b(x,u,p)=\bar{a}(\bar{x})$  and
\begin{equation*}
\begin{array}{lll}
I_1=-f_q,&
I_2=-\frac{2}{9}I_1^2-f_p-\frac{1}{3}\hat{D}_xI_1,\\
I_3=J_3^3=-\frac{1}{3}I_1I_2-f_u-\frac{1}{2}\hat{D}_xI_2,&
I_5=\frac{1}{3J_3^2}(I_1J_3+3\hat{D}_xJ_3),\\
 I_7=-I_{5_q}J_3^2,&
I_{9}=2J_3\hat{D}_xI_5-I_2+J_3^2I_5^2,\\
Q=-\frac{I_1I_7+3J_3^2I_{5_p}+3J_{3_u}-3J_3I_4\hat{D}_xI_5}{3J_3},&
K=\frac{I_9}{J_3^2}.
\end{array}
\end{equation*}
\proof
The equivalence of the third-order  ODEs (\ref{ee4.8})  under the contact transformation (\ref{ee4.9}) can be checked  using the  invariant five-dimensional  coframe  (\ref{2.35}) on the space $M^{(1)}=M\times G_5$ such that
\begin{equation}\label{eee4.10}
\begin{aligned}
&\Phi^*\left(\begin{array}{ccccc}
\bar{a}_1 &0 & 0&0&0 \\
\bar{a}_1\bar{I}_5 &\frac{\bar{a}_1}{\bar{J}_3}  & 0 &0&0\\
(\frac{\bar{I}_5^2}{2}+\frac{\bar{I}_2}{2\bar{J}_3^2})\bar{a}_1&(\frac{\bar{I}_5}{\bar{J}_3}+\frac{\bar{I}_1}{3\bar{J}_3^2})\bar{a}_1  & \frac{\bar{a}_1} {\bar{J}_3^2}&0&0\\
\bar{I}_7&\bar{I}_4&0&\bar{J}_3&0\\
-\frac{\bar{I}_1\bar{I}_7+3\bar{J}_3^2\bar{I}_{5_{\bar{p}}}+3\bar{J}_{3_{\bar{u}}}-3\bar{J}_3I_4\hat{D}_{\bar{x}}\bar{I}_5}{3\bar{J}_3}&\frac{\bar{I}_7}{\bar{J}_3}-\bar{I}_4\bar{I}_5&0&-\bar{I}_5\bar{J}_3&\frac{1}{\bar{a}_1}
\end{array}\right) \left(\begin{array}{c}
 \bar{\omega}^1\\\bar{\omega}^2\\\bar{\omega}^3\\\bar{\omega}^4 \\
d\bar{a}_1\\
\end{array}\right)=\\
&\left(\begin{array}{ccccc}
a_1 &0 & 0&0&0 \\
a_1I_5 &\frac{a_1}{J_3}  & 0 &0&0\\
(\frac{I_5^2}{2}+\frac{I_2}{2J_3^2})a_1&(\frac{I_5}{J_3}+\frac{I_1}{3J_3^2})a_1  & \frac{a_1} {J_3^2}&0&0\\
I_7&I_4&0&J_3&0\\
-\frac{I_1I_7+3J_3^2I_{5_p}+3J_{3_u}-3J_3I_4\hat{D}_xI_5}{3J_3}&\frac{I_7}{J_3}-I_4I_5&0&-I_5J_3&\frac{1}{a_1}
\end{array}\right)\left(\begin{array}{c}
 \omega^1\\\omega^2\\\omega^3\\\omega^4\\
da_1\\
\end{array}\right),
\end{aligned}
\end{equation}
where $\Phi^*$ is the pullback arising from the first prolongation of the contact transformation  (\ref{ee4.9}).\\

 The function $\bar{a}(\bar{x})$ can be evaluated as follows:
inserting the contact transformation (\ref{ee4.9}) into the invariant relation $\Phi^*\bar{\theta^4}=\theta^4$ and using $~\bar{J_3}=-\bar{a}(\bar{x})$ for $\bar{u}^{\prime \prime\prime}=\bar{a}^3(\bar{x}) \bar{u}$
 results in $D_x\varphi=-\frac{J_3}{\bar{a}(\bar{x})}$.
Thus, the first equation in the system (\ref{2.46}) can be rewritten as 
\begin{equation}
\begin{array}{lll}
K=-\frac{2}{J_3}\hat{D}_xH+H^2,
\end{array}
\end{equation}
 where  $ H(x,u,p)=\frac{\hat{D}_{\bar{x}}\bar{a}(\bar{x})}{\bar{a}^2(\bar{x})}$. Since $D_x\varphi=-\frac{J_3}{\bar{a}(\bar{x})}$ and $\hat{D}_x=(\hat{D}_x\varphi)\hat{D}_{\bar{x}}$, then 
\begin{equation}
\begin{array}{lll}
 \hat{D}_xb=-(J_3H)b,
\end{array}
\end{equation}
where $b(x,u,p)=\bar{a}(\bar{x})$.\\
Finally, incorporating  the values 
$$\bar{I_1}=0,~\bar{I_2}=0,~\bar{J_3}=-\bar{a}(\bar{x}) ,~\bar{I_4}=0,~\bar{I_5}=-\frac{\bar{a}'(\bar{x}) }{\bar{a}^2(\bar{x}) },~\bar{I_7}=0,~\bar{Q}=0$$
for $\bar{u}^{\prime \prime\prime}=\bar{a}^3(\bar{x}) \bar{u}$ and $\bar{a_1}=1$, in equation (\ref{eee4.10}) yields
\begin{equation}\label{eeee4.13}
\begin{array}{lll}
\left( {\begin{array}{*{20}c}
\bar{\omega}^1\\\bar{\omega}^2\\\bar{\omega}^3\\\bar{\omega}^4\\ 0\\
\end{array}} \right) 
&=\left( {\begin{array}{*{20}c}
   {a_1 } & 0 & 0 &0&0 \\
  - (H+I_5)ba_1&-\frac{ba_1}{J_3}& 0&0&0  \\
  (\frac{(H+I_5)^2}{2}+\frac{I_2}{2J_3^2})b^2a_1 &(\frac{H+I_5}{J_3}+\frac{I_1}{3J_3^2})b^2a_1& \frac{b^2a_1}{J_3^2}&0&0  \\
-\frac{I_7}{b}&-\frac{I_4}{b}&0&-\frac{J_3}{b}&0\\
Q-HI_7&\frac{I_7}{J_3}-I_4(H+I_5)&0&-J_3(H+I_5)&\frac{1}{a_1}
\end{array}} \right)
 \left( {\begin{array}{*{20}c} 
\omega^1\\\omega^2\\\omega^3\\\omega^4\\
da_1\\
\end{array}} \right).
\end{array}
\end{equation}
Computing the pullback associated with the left-hand side of  equation (\ref{eeee4.13}) gives systems (\ref{4.12}), (\ref{4.13}), and(\ref{4.14}).\\
\endproof
\end{proposition} 

\begin{example}
Let us consider the canonical form
\begin{equation}\label{nn2}
	u'''=\frac{\alpha u''^2}{u'}.
\end{equation}
It can be readily verified that the function $f(x,u,p,q)=\frac{\alpha q^2}{p}$ fulfills the condition stated in Theorem \ref{Th3.1}. This establishes its equivalence to the canonical form $\bar{u}'''=s\bar{u}'+\bar{u}$,
 with five symmetries under contact transformation, where $s=\frac{3\alpha^2-9\alpha+9}{(2\alpha^3-9\alpha^2+9\alpha)^{2/3}}.$\\
 
Using $I_1=-\frac{2\alpha q}{p},~I_2=-\frac{\alpha q^2 (2\alpha-3)}{9p^2},~J_3=\frac{(2\alpha^3-9\alpha^2+9\alpha)^{\frac{1}{3}}}{3p}q,~I_4=\frac{(2\alpha^3-9\alpha^2+9\alpha)^{\frac{1}{3}}}{3p}$,
~$I_5=\frac{(\alpha-3)}{(2\alpha^3-9\alpha^2+9\alpha)^{\frac{1}{3}}},~I_7=0,~Q=0$,  evaluated for equation (\ref{nn2}), we obtain the transformation through the following steps:
\begin{description}
\item[Step 1] The auxiliary funcion $a_1=p^{\frac{\alpha}{3}-1},$ is a solution for the system (\ref{4.3}).
\item[Step 2] The solution to system (\ref{4.4}) is 
$\varphi =-\frac{1}{3}(2\alpha^3-9\alpha^2+9\alpha)^{1/3}\ln{p}$.
\item[Step 3]The solution to system (\ref{4.5}) is
$\chi=\frac{\alpha p x-\alpha u+3u}{(2\alpha^3-9\alpha^2+9\alpha)^{\frac{1}{3}}}p^{\frac{\alpha}{3}-1}, \psi=(u-xp)p^{\frac{\alpha}{3}-1}$.
\end{description}
Hence, the proper contact transformation
\begin{equation}
\begin{array}{lll}
\varphi \left( x,u,p \right)=-\frac{1}{3}(2\alpha^3-9\alpha^2+9\alpha)^{1/3}\ln{p},~\psi \left( x,u,p \right)=(u-xp)p^{\frac{\alpha}{3}-1},
\end{array}
\end{equation}
transforms the canonical form $\bar{u}'''=s\bar{u}'+\bar{u}$ to the canonical form (\ref{nn2}).
\end{example}
\begin{remark}
The function $f(x,u,p,q)=\frac{\alpha q^2}{p}$  does not satisfy the conditions of  Theorem 2.1 in \cite{Dweik2019}. So the canoical form $u'''=\frac{\alpha u''^2}{u'}$ is not equivalent to the canonical form $\bar{u}'''=s\bar{u}'+\bar{u}$,
 under point transformation.
\end{remark}

\begin{example}
Let us consider a class of third-order nonlinear ODE
\begin{equation}\label{4.29}
	u'''=-xu'^4 u''^3+uu'^3u''^3.
\end{equation}
The function $f(x,u,p,q)=-xp^4q^3+up^3q^3$ fulfills the condition stated in Theorem \ref{Th3.2} and is thus equivalent to the canonical form $\bar{u}'''=\bar{a}^3(\bar{x})\bar{u},$ with four symmetries under contact transformation.\\

Using $I_1=3p^3q^2(px-u),~I_2=0,~J_3=-pq,~I_4=-p$,
~$I_5=-\frac{1}{p^2},~I_7=0,\\~Q=0,~K=-\frac{3}{p^4}$,  evaluated for equation (\ref{4.29}), we obtain the transformation through the following steps:
\begin{description}
\item[Step 1] The auxiliary funcion $H=\frac{1}{p^2},$ is a solution for  system (\ref{4.10}).
\item[Step 2] The auxiliary funcion $b=p,$ is a solution for  system (\ref{4.11}).
\item[Step 3] The auxiliary funcion $a_1=1,$ is a solution for  system (\ref{4.12}).
\item[Step 4] The solution to system(\ref{4.13}) is
$\varphi =p$.
\item[Step 5] The solution to system(\ref{4.14}) is
$\chi=-x,~\psi=u-px$.
\end{description}
Moreover, the function $\bar{a}(\bar{x})$ of the resulting canonical form is given by  $b(x,u,p)=\bar{a}(\bar{x})$ as $\bar{a}(\bar{x})=\bar{a}(p)=p$. Therefore, $\bar{a}(\bar{x})=\bar{x}$ and the canonical form $\bar{u}'''=\bar{x}^3\bar{u}$ can be obtained for the ODE (\ref{4.29}) via the contact transformation 
\begin{equation}
\varphi \left( x,u,p \right)=p,~\psi\left( x,u,p \right)=u-xp,~\chi\left( x,u,p \right)=-x.
\end{equation}
\end{example}
\begin{remark}
The function $f(x,u,p,q)=-xp^4q^3+up^3q^3$  does not satisfy the conditions of  Theorem 3 in \cite{Dweik2018_2}. So the ODE $u'''=-xu'^4 u''^3+uu'^3u''^3$ is not equivalent to the canonical form $\bar{u}'''=\bar{a}(\bar{x})^3\bar{u}$,
 under point transformation.
\end{remark}
\section{Conclusion}
A central contribution of this work is a novel framework for implementing Cartan’s method. By systematically branching through relative invariants and simplifying structures via the introduction of a chain of auxiliary functions, we provide an efficient way to manage the complexity and expression growth inherent in Cartan’s method.

Using this framework, we study  the equivalence problem of linearizable third-order ODEs admitting four- and five-dimensional Lie symmetry algebras under contact transformations. The proposed approach yields invariant coframes that fully characterize the corresponding canonical forms and provide a clear geometric distinction between the two symmetry cases. 

In addition, we provide a constructive procedure for obtaining contact transformations that reduce a given equation to its linear canonical form. This procedure is explicitly realized through the integration of associated systems of linear or Riccati PDEs derived from the invariant coframes. Unlike earlier approaches based on direct solution of nonlinear PDE systems, our method is algorithmic in nature.

This work extends previous results obtained under point transformations to more general setting of contact transformations. The framework developed here provides an effective computational pathway for constructing the corresponding transformations.
\subsection*{Acknowledgements}
The authors would like to express their gratitude to Birzeit University for the support and facilities provided during this research. FM thanks Wits for support.\\\\
The authors declare that no competing interests exist.


\begin{thebibliography}{99}
\bibitem{Lie}  {Lie S. Klassifikation und integration von gewönlichen differentialgleichungen zwischen $x, y$, die eine Gruppe von Transformationen gestaten. Arch Math 1883; VIII(IX): 187.}
\bibitem{Lie*} {Lie S. Arch Mat Nat 1883; 8:371427 . (Reprinted in Lies Gessammelte Abhandlundgen, 5, 1924, paper XIY, pp 362427.}
\bibitem{Mahomed} {Mahomed F. M. Point symmetry group classification of ordinary differential equations: a survey of some results. Mathematical Methods in the Applied Sciences (2007) 30: 1995–2012.}
\bibitem{Lie1}  Lie S., Over en classe geometriske Transformationer, Doctoral Thesis, University of Christiana, 1871.
\bibitem{Lie2}  Lie S., Begrndung einer Invariantentheorie der Berhrungstransformationen, Mathematische Annalen 8, 1874, 215288.
\bibitem{Lie3}  Lie S. and Engel, F., Theorie der Transformationsgruppen, B. G. Teubner, Leipzig, Vol. 1, 1888.
\bibitem{Lie4}  Lie S. and Engel, F., Theorie der Transformationsgruppen, B. G. Teubner, Leipzig, Vol. 2, 1890.
\bibitem{Lie5} Lie S. and Engel, F., Theorie der Transformationsgruppen, B. G. Teubner, Leipzig, Vol. 3, 1893.
\bibitem{Lie6} Lie S. and Scheffers, G., Vorlesugen ber Differentialgleichungen mit bekanten infinitesimalen Transformationen, B. G. Teubner, Leipzig, 1891.
\bibitem{Yumaguzhin1997} Yumaguzhin, Valeriy A. Contact classification of 3rd-order linear ODEs', The Diffeity Institute Preprint Series. 1997.
\bibitem{Yumaguzhin1996} Yumaguzhin, Valeriy A. "Classification of 3rd order linear ODE up to equivalence." Differential Geometry and its Applications 6, no. 4 (1996): 343-350.
\bibitem{Soh2002} Wafo Soh, Célestin,  Mahomed,  Fazal M.  and  Qu. C. "Contact symmetry algebras of scalar ordinary differential equations." Nonlinear Dynamics 28 (2002): 213-230.
\bibitem{Svi1} Svishchevskii, S. R., Lie-B\"acklund symmetries of linear ODEs and invariant linear spaces, in Modern Group Analysis, G. N. Yakovenko (ed.), Institute for Mathematical Modelling, Russian Academy of Sciences, Moscow, 1993, pp. 324.
\bibitem{Svi2}  Svishchevskii, S. R., Lie-B\"acklund symmetries of linear ODEs and generalized separation of variables in nonlinear equations, Physics Letters A 199, 1995, 344-348.
\bibitem{Ibragimov1977} Ibragimov, N. H.,  Khalique, C. M. and  Mahomed, F. M. "All linear ordinary differential equations admitting contact symmetries." In Proceedings of the International Conference at the Sophus Lie Centre, pp. 155-159. Mars Publishers, Symmetri Foundation, Trondheim, 1997.
\bibitem{Mahomed1990} Mahomed,  Fazal M. and  Leach,  P. G. L. "Symmetry Lie algebras of nth order ordinary differential equations." Journal of Mathematical Analysis and Applications 151, no. 1 (1990): 80-107.
\bibitem{Chern1940} Chern, S. S.  The geometry of the differential equation $y'''=F(x,y,y,y'')$, Sci. Rep. Nat. Tsing Hua Univ. 4 (1940), 97-111.
\bibitem{Neut2002} Neut, Sylvain and  Petitot Michel. "La géométrie de l'équation $y'''= f (x, y, y', y'')$." Comptes rendus. Mathématique 335, no. 6 (2002): 515-518.
\bibitem{Wnschmann1905}  W$\ddot{\textrm{u}}$nschmann, K. $\ddot{\textrm{U}}$ber Beruhrungsbedingungen bei Differentialgleichchungen, Enzyklop$\ddot{\textrm{a}}$die der Math. Wiss. {\bf3}, (1905), 490-492.
\bibitem{Ibra} Ibragimov, N. H. and Meleshko,  V. S. Linearization of third-order ordinary differential equations by point and contact transformations. J Math Anal Appl 2005;308:266–89 .
\bibitem{Dweik2019} Al-Dweik  Ahmad, Y.,  Mahomed, Fazal M.,and  Mustafa, Muhammad T. "Invariant characterization of third-order ordinary differential equations $u'''= f (x, u, u', u'')$ with five-dimensional point symmetry group." Communications in Nonlinear Science and Numerical Simulation 67 (2019): 627-636.
\bibitem{Dweik2018_2} Al-Dweik, Ahmad Y.,  Mustafa, M. T.,  Mahomed, Fazal M. and Rajai, S. Alassar. "Linearization of third‐order ordinary differential equations via point transformations." Mathematical Methods in the Applied Sciences 41, no. 16 (2018): 6955-6967.
\bibitem{Olver1995} Olver, P. J. Equivalence, Invariants and Symmetry, Cambridge University Press, Cambridge, 1995.                           
\bibitem{Neut2003} Neut, S.  Implantation et nouvelles applications de la méthode d'équivalence de Cartan. Phd thesis, Univ. Lille I, 2003.
\end{thebibliography}
\end{document}